\documentclass[10pt]{amsart}
\usepackage[bb=fourier,cal=pxtx]{mathalfa}
\usepackage{amssymb}
\usepackage{mathtools}      
\usepackage{bm}             
\usepackage{indentfirst}    
\usepackage{amsthm}
\usepackage{thmtools}
\usepackage{enumitem}       
\usepackage[colorlinks=true,backref=page]{hyperref} 
\usepackage[font=small]{caption}
\usepackage{tikz}           
\usepackage{tikz-cd}        
\usetikzlibrary{decorations.pathreplacing}

\renewcommand*{\backref}[1]{}
\renewcommand*{\backrefalt}[4]{\quad \tiny
  \ifcase #1 (\textbf{NOT CITED.})%
  \or    (Cited on page~#2.)%
  \else   (Cited on pages~#2.)%
  \fi}

\makeatletter

\def\MRbibitem{\@ifnextchar[\my@lbibitem\my@bibitem}

\def\mybiblabel#1#2{\@biblabel{{\hyperref{http://www.ams.org/mathscinet-getitem?mr=#1}{}{}{#2}}}}

\def\myhyperanchor#1{\Hy@raisedlink{\hyper@anchorstart{cite.#1}\hyper@anchorend}}

\def\my@lbibitem[#1]#2#3#4\par{%
  \item[\mybiblabel{#2}{#1}\myhyperanchor{#3}\hfill]#4%
  \@ifundefined{ifbackrefparscan}{}{\BR@backref{#3}}%
  \if@filesw{\let\protect\noexpand\immediate
    \write\@auxout{\string\bibcite{#3}{#1}}}\fi\ignorespaces%
}

\def\my@bibitem#1#2#3\par{%
  \refstepcounter\@listctr
  \item[\mybiblabel{#1}{\the\value\@listctr}\myhyperanchor{#2}\hfill]#3%
  \@ifundefined{ifbackrefparscan}{}{\BR@backref{#2}}%
  \if@filesw\immediate\write\@auxout
    {\string\bibcite{#2}{\the\value\@listctr}}\fi\ignorespaces%
}

\makeatother
\declaretheorem{theorem}

\declaretheorem[numberwithin=section]{otherthm}
\declaretheorem[sibling=otherthm]{lemma}
\declaretheorem[sibling=otherthm]{corollary}
\declaretheorem[sibling=otherthm]{proposition}

\declaretheorem[sibling=otherthm,style=remark]{remark}
\declaretheorem[name=Acknowledgements, style=remark, numbered=no]{ack}
\declaretheorem[, name=Funding, style=remark, numbered=no]{fund}

\hypersetup{bookmarksdepth = 3} 
\numberwithin{equation}{section}     

\setlist[enumerate,1]{label={\upshape(\alph*)},ref=\alph*}
\setlist[enumerate,2]{label={\upshape(\arabic*)},ref=\arabic*}

\newcommand{\R}{\mathbb{R}}
\newcommand{\Z}{\mathbb{Z}}
\newcommand{\N}{\mathbb{N}}

\renewcommand{\P}{\mathbb{P}}
\newcommand{\cC}{\mathcal{C}}

\newcommand{\cG}{\mathcal{G}}\newcommand{\cH}{\mathcal{H}}

\newcommand{\cM}{\mathcal{M}}

\newcommand{\cS}{\mathcal{S}}\newcommand{\cU}{\mathcal{U}}

\newcommand{\st}{\;\mathord{;}\;}

\newcommand{\disagree}{\mathbin{\dagger}}
\DeclareBoldMathCommand\ba{a}
\DeclareBoldMathCommand\bb{b}

\newcommand{\mediumint}{\int}
\newcommand{\dd}{\,\mathrm{d}}   
\DeclareMathOperator{\Lip}{Lip}
\DeclareMathOperator{\var}{var}

\DeclareMathOperator{\gap}{gap}
\DeclareMathOperator{\ergsup}{erg\,sup}
\DeclareMathOperator{\per}{per}
\DeclareMathOperator{\supp}{supp}
\newcommand{\lock}{\mathsf{L}_\ba}
\newcommand{\Leb}{\mathrm{Leb}}

\newcommand*\circled[1]{\tikz[baseline=(char.base)]{
    \node[shape=circle,draw,inner sep=1pt] (char) {\footnotesize{#1}};}}

\newcommand{\arxiv}[1]{Preprint \href{http://arxiv.org/abs/#1}{arXiv:{#1}}}

\renewcommand{\epsilon}{\varepsilon}
\renewcommand{\phi}{\varphi}
\renewcommand{\setminus}{\smallsetminus}
\renewcommand{\emptyset}{\varnothing}

\makeatletter
\DeclareFontFamily{OMX}{MnSymbolE}{}
\DeclareSymbolFont{MnLargeSymbols}{OMX}{MnSymbolE}{m}{n}
\SetSymbolFont{MnLargeSymbols}{bold}{OMX}{MnSymbolE}{b}{n}
\DeclareFontShape{OMX}{MnSymbolE}{m}{n}{
    <-6>  MnSymbolE5
   <6-7>  MnSymbolE6
   <7-8>  MnSymbolE7
   <8-9>  MnSymbolE8
   <9-10> MnSymbolE9
  <10-12> MnSymbolE10
  <12->   MnSymbolE12
}{}
\DeclareFontShape{OMX}{MnSymbolE}{b}{n}{
    <-6>  MnSymbolE-Bold5
   <6-7>  MnSymbolE-Bold6
   <7-8>  MnSymbolE-Bold7
   <8-9>  MnSymbolE-Bold8
   <9-10> MnSymbolE-Bold9
  <10-12> MnSymbolE-Bold10
  <12->   MnSymbolE-Bold12
}{}

\let\llangle\@undefined
\let\rrangle\@undefined
\DeclareMathDelimiter{\llangle}{\mathopen}%
                     {MnLargeSymbols}{'164}{MnLargeSymbols}{'164}
\DeclareMathDelimiter{\rrangle}{\mathclose}%
                     {MnLargeSymbols}{'171}{MnLargeSymbols}{'171}
\makeatother

\begin{document}

\title{Ergodic optimization of prevalent super-continuous
functions}
\date{\today}

\author{Jairo Bochi}
\author{Yiwei Zhang}



\begin{abstract}
Given a dynamical system, we say that
a performance function has property~P if its time averages along orbits are
maximized at a periodic orbit. It is conjectured by several authors
that for sufficiently hyperbolic dynamical systems, property~P
should be typical among sufficiently regular performance functions.
In this paper we address this problem using a
probabilistic notion of typicality that is suitable to infinite
dimension: the concept of prevalence as introduced by Hunt, Sauer,
and Yorke. For the one-sided shift on two symbols, we prove that
property~P is prevalent in spaces of functions with a strong modulus
of regularity. Our proof uses Haar wavelets to approximate the
ergodic optimization problem by a finite-dimensional one, which can
be conveniently restated as a maximum cycle mean problem on a
de~Bruijin graph.
\end{abstract}

\maketitle

\section{Introduction}

\subsection{Are optimal orbits typically periodic?}

Consider a continuous transformation $T \colon X \to X$
on a compact metrizable space,
and a continuous ``performance'' function $f \colon X \to \R$.
The main questions of \emph{ergodic optimization} are:
Along which orbits of $T$ are the time averages of $f$ largest?
What is the statistics of those orbits?
Equivalently, for which invariant probability measures $\mu$ is average performance $\int f \dd \mu$ maximal?
What properties do these so-called \emph{maximizing measures} have?
What happens as $f$ changes?
The subject was nicely surveyed by Jenkinson \cite{Jenkinson_survey}.
Some maximizing measures of special interest in thermodynamic formalism are
the so-called \emph{ground states}: see~\cite{BLL}.

If the dynamical system $T$ has few invariant measures then the questions above
become uninteresting.
So it is natural to focus the study on chaotic systems, as for example
uniformly hyperbolic or uniformly expanding ones.
These systems possess many invariant probability measures,
the simplest of which are the \emph{periodic measures}
(i.e.\ those supported on periodic orbits).
Let us say that the performance function
satisfies \emph{property $\mathsf{P}$} with respect the dynamical system $T$
if it has a periodic maximizing measure.

Based on experimental results and heuristic arguments,
Hunt and Ott \cite{HO} conjectured that for \emph{typical} chaotic systems and \emph{typical} parameterized families of smooth functions, property~$\mathsf{P}$ holds for \emph{most} values of the parameter. Here ``most'' means full Lebesgue measure in the parameter space (but the meanings of ``typical'' are not specified).
One of the examples studied by these authors consisted on the doubling map
$T(\theta) = 2\theta$ on the circle $\R/2\pi\Z$
together with the family of functions formed by linear combinations of cosines and sines. For this specific case, the conjecture was later confirmed by Bousch \cite{Bousch_poisson}.
(See also \cite{Jenkinson_fish}.)

A different conjecture by Yuan and Hunt \cite{YH} is that for every uniformly hyperbolic or uniformly expanding map, property~$\mathsf{P}$ holds for \emph{topologically generic} $C^1$ performance functions, i.e., it holds for every function in a dense $G_\delta$ subset.
As these authors observe, the interior of property~$\mathsf{P}$ is dense in itself, so if the conjecture is true then property~$\mathsf{P}$ actually holds for an open and dense subset of the space of $C^1$ functions.

Regularity assumptions are important.
Assume that $T$ is a ``hyperbolic'' (e.g., expanding) map.
Then maximizing measures of topologically generic continuous function have full support (see \cite{Bousch_Walters,Jenkinson_survey,Morris_C0});
in particular the Yuan--Hunt conjecture fails dramatically in the space of continuous functions.
On the other hand, every sufficiently regular function (e.g., a H\"older one)
has the following agreeable property: there exists a compact invariant set $K$
such that an invariant probability measure $\mu$ is maximizing if and only if $\mu(K)=1$.
This fact and related ones (especially the so-called Ma\~n\'e lemma) were discovered independently by many authors in various settings: see \cite{Bousch_Walters,Bousch_Mane} and references therein.

Much research was done about property~$\mathsf{P}$
and related ones on spaces of sufficiently regular functions: \cite{YH,CLT,Bousch_Walters,BQ,Morris_entropy,QS}.
We highlight the recent work of Contreras~\cite{Contreras_EO},
who showed that this property is open and dense among Lipschitz functions
when the dynamics is expanding.

Let us also mention that many of the problems and results we have alluded to
were originally formulated in the context of classical mechanics,
where optimization is naturally related with the principle of least action: see \cite{CI} and references therein.
There is a celebrated conjecture by Ma\~{n}\'{e}~\cite{Mane}
on the genericity of property~$\mathsf{P}$ in those contexts,
which was proved very recently by Contreras~\cite{Contreras_Mane}.

\smallskip

In this paper we further investigate the genericity of property~$\mathsf{P}$
in abstract ergodic optimization,
not in the topological sense but in a \emph{probabilistic} one,
thus in a spirit more akin to \cite{HO}.
The first step is to define ``probabilistic genericity'' in a space of functions.
Though a possible approach for this is suggested in \cite{HO},
we follow another route which seems more appropriate.

Probabilistic genericity should have something to do with having full measure.
Unfortunately, topological vector spaces do not have any natural reference measures,
except when they are finite-dimensional.
Worse, there is no natural equivalence class of measures.
On the positive side, what does exist (under minor technical assumptions)
is a natural class of ``fat'' subsets,
which corresponds in the finite-dimensional case to the class of subsets of full Lebesgue measure.
This is the class of \emph{prevalent} sets
discovered by Hunt, Sauer, and Yorke \cite{HSY},
though it was previously introduced in a slightly less general setting by Christensen \cite{Christ}. The surveys \cite{OY,HK} contain much information on the subject, including generalizations, applications, and open problems in many areas -- especially Dynamical Systems.

In general, prevalence and topological genericity are neither one stronger than the other;
indeed the references above contain examples of many properties of natural interest that are generic in one sense but not in the other.
However, for properties known to be open (or to contain an open dense subset), topological genericity means only denseness, which is always weaker than prevalence.

Thus we are inevitably led to the the following question:
Does a prevalent version of the theorem of Contreras~\cite{Contreras_EO} hold?
In other words, if the dynamics is expanding,
is property~$\mathsf{P}$ prevalent among Lispchitz functions?

Though being currently unable to solve the problem in full generality,
we do answer the question positively under further assumptions.
First, we deal with the simplest interesting dynamics in this context:
the full one-sided shift on two symbols.
Second, we work on spaces of \emph{super-continuous} functions,
i.e.\ functions that have a very strong modulus of regularity.
The setting comes from the work of Quas and Siefken~\cite{QS},
who proved denseness of property~$\mathsf{P}$ in some of these spaces.

Let us proceed with the exact definitions and statements.

\subsection{Precise statements}\label{ss.precise}

Let $V$ be a completely metrizable
vector space. Following \cite{HSY}, we say that a Borel set $S
\subset V$ is \emph{shy} if there exists a compactly supported Borel
probability measure $\mu$ such that
$$
\mu(S + v) = 0 \quad \text{for every $v\in V$.}
$$
In that case we say that $\mu$ is \emph{transverse} to $S$.\footnote{If the space $V$ is separable then the compactness hypothesis is basically avoidable. However we will deal with non-separable spaces.}

It is easy to see that if a set $S$ is shy then it has measures transverse to it with support of arbitrarily small diameter.
Therefore, given any element of $V$ it is possible to add to it a small random perturbation so that it falls outside of $S$ with probability $1$.

A set whose complement is contained in a shy set is called a \emph{prevalent} set.
Prevalence has the following properties (see \cite{HSY}):
it is preserved under translation, augmentation, and countable intersections,
it implies denseness, and it coincides with the notion of having full Lebesgue measure
when the ambient space is finite-dimensional.


\medskip

Next, let us state more formally the setting for ergodic optimization.
Let $(X,d)$ be a compact metric space, and $T \colon X \to X$ be a continuous map.
Let $\cM_T$ denote the set of $T$-invariant Borel probability measures.
If $\mu \in \cM_T$ is supported on a periodic orbit then it is called a \emph{periodic measure}.
Let $C^0(X)$ denote the space of real-valued continuous functions on $X$.
The \emph{ergodic supremum} of $f \in C^0(X)$ is defined as
\begin{equation}\label{e.ergsup}
\ergsup (f) \coloneqq \sup_{\mu \in \cM_T} \langle f , \mu \rangle \, ,
\end{equation}
where angle brackets denote integration.
A measure $\mu \in \cM_T$ where this supremum is attained is called
a \emph{maximizing measure} for $f$.
By continuity and compactness, maximizing measures always exist.
The ergodic supremum can also be expressed in terms of maximal time averages of $f$: see \cite{Jenkinson_survey}.
If one of these is periodic then we say that $f$ has \emph{property~$\mathsf{P}$.}

Let $C^{\Lip}(X) \subset C^0(X)$ be the set of Lipschitz functions,
endowed with the \emph{Lipschitz norm}
$\| \mathord{\cdot} \|_{\Lip} \coloneqq \| \mathord{\cdot} \|_{\infty}
+ \Lip( \mathord{\cdot})$,
i.e.\ the sum of the supremum norm and the least Lipschitz constant.
Then $C^{\Lip}(X)$ is a Banach space,
and it is nonseparable unless $X$ is countable,
because it contains the discrete subset $\{d(x,\mathord{\cdot}) \st x\in X\}$.

We say that a function $f \in C^{\Lip}(X)$
has the \emph{locking\footnote{This is a free translation of the \emph{verrouillage} term used in \cite[\S 8]{Bousch_Walters}.} property} if
it has a unique maximizing measure $\mu$
that is periodic (in particular, $f$ has property~$\mathsf{P}$)
and moreover $\mu$ is the unique maximizing measure
for every function sufficiently close to $f$ in the Lipschitz norm.
We can also say that $f$ and all functions close to it are \emph{locked} to the measure $\mu$.
Of course, the locking property depends on the metric $d$.

The set of functions with the locking property is exactly the interior of the set
of functions with property~$\mathsf{P}$:
one inclusion is obvious from the definition,
and the other was shown by Yuan and Hunt \cite[Remark~4.5]{YH}\footnote{Actually this exact statement is not proved in the paper \cite{YH} but it follows by an adaptation of the arguments. The interested reader may try to figure out the details himself or check our note \href{http://www.mat.uc.cl/~jairo.bochi/docs/lock.pdf}{www.mat.uc.cl/$\sim$jairo.bochi/docs/lock.pdf}.}.
In particular, the locking property is dense among the functions with property~$\mathsf{P}$.
We will henceforth focus our attention on the locking property.

\smallskip

We work in the symbolic space
$\Omega \coloneqq \{0,1\}^\N$, where $\N$ includes $0$.
Suppose $\ba = (a_n)_{n \in \N}$ is a strictly decreasing sequence
of positive numbers converging to zero, or in symbols $\ba \searrow 0$.
Then we define a metric $d_\ba$ (actually an ultrametric) on $\Omega$ as follows:
given two sequences $x = (x_n)$, $y = (y_n) \in \Omega$, if they are different we set
$d_\ba(x,y) \coloneqq a_m$ where $m$ is their first disagreement
(i.e., the least position where the symbols differ),
otherwise we set $d_\ba(x,y) \coloneqq 0$.
All metrics of this form induce the product topology on $\Omega$.

Let $C^\ba(\Omega) \coloneqq C^{\Lip}(\Omega)$ where $\Omega$ is
endowed with the metric $d_\ba$; in other words $C^\ba(\Omega)$ is
the space of functions on $\Omega$ that are Lipschitz with respect
to the metric $d_\ba$. As explained above, this is a non-separable Banach space.

Of course, the faster the sequence $\ba$ decays, the more regular are the functions in $C^\ba(\Omega)$.
As an example, consider geometric sequences $(\theta^n)_{n \in \N}$
with factor $\theta \in (0,1)$;
then $f$ is $\alpha$-H\"older with respect to the metric $d_{(\theta^n)}$
if and only if $f \in C^{(\theta^{\alpha n})}(\Omega)$.

\smallskip

Let us consider ergodic optimization on the function spaces $C^\ba(\Omega)$
with respect to the shift dynamics $\sigma \colon \Omega \to \Omega$.
Let
$$
\lock \subset C^{\ba}(\Omega)
$$
denote the set of $d_\ba$-Lipschitz functions with the locking property.

\smallskip

It was proved by Quas and Siefken~\cite{QS} that if the sequence $\ba = (a_n)_{n \in \N}$
is such that $\lim a_{n+1}/a_n = 0$ then $\lock$ is
dense (and thus topologically generic) in $C^\ba(\Omega)$.
Contreras~\cite{Contreras_EO} obtains the same conclusion under the weaker condition
$\sup a_{n+1}/a_n < 1$.
(Actually, Quas and Siefken deal with one-sided shifts in an arbitrary finite number of symbols,
while Contreras deals with arbitrary expanding maps.)

Our aim is to replace denseness by prevalence in this kind of results.
However we need an even stronger modulus of regularity.
Let us say that a sequence $\ba = (a_n)_{n \in \N}$ is \emph{evanescent}
if $\ba \searrow 0$ and
\begin{equation}\label{e.eva}
\frac{a_{n+1}}{a_n} = O \left( 2^{-2^{n+2}} \right) \, .
\end{equation}
In this paper, we prove the following:

\begin{theorem}\label{t.prevalence}
If $\ba$ is an evanescent sequence then $\lock$ is a prevalent set in $C^\ba(\Omega)$.
\end{theorem}

Since the set $\lock$ is open by definition,
and every prevalent set is dense,
we reobtain its topological genericity under the strong assumption~\eqref{e.eva}.
We stress that our methods are very different from those of \cite{QS,Contreras_EO},
however.

Currently we are not able to weaken the evanescence condition \eqref{e.eva} considerably.
The difficulties will be discussed later.
On the other hand, our methods should be easily extendable
to shifts on an arbitrary finite number of symbols
or more generally to subshifts of finite type.
Favoring clarity at the expense of generality,
we decided to work with shifts on $2$ symbols only.
This choice has the additional advantage of
being able to express the random perturbations leading to the proof of
Theorem~\ref{t.prevalence} in a quite convenient way.

\subsection{The random perturbations}

Let us describe explicitly certain probability measures that are transverse to the complement of
$\lock$.
Therefore adding to any given function in $C^\ba(\Omega)$
a random perturbation chosen according to one of those probability measures
one almost surely obtains a function in $\lock$.

\smallskip

Let $\Omega^*$ denote the set of all finite words in the letters $0$, $1$, including the empty word.
Each $\omega \in \Omega^*$ determines a \emph{cylinder} $[\omega]$
which is the set of elements of $\Omega$ that have $\omega$ as an initial subword.
The \emph{level} of the cylinder $[\omega]$ is defined as the length of the word $\omega$,
which is denoted by $|\omega|$.

For each $\omega \in \Omega^*$, 
we define the \emph{Haar function}:
\begin{equation}\label{e.Haar_function}
h_\omega \coloneqq \frac{\chi_{[\omega0]} - \chi_{[\omega 1]}}{2} \, ,
\end{equation}
where $\chi$ denotes characteristic functions and $\omega 0$, $\omega 1$ denote concatenated words.
These functions are a straightforward adaptation of the classical Haar functions on $[0,1]$ (see e.g.\ \cite[\S6.3]{Pinsky}).
Working on the Cantor set $\Omega$ has an advantage over working on the interval $[0,1]$:
our Haar functions \eqref{e.Haar_function} are continuous.
Our unusual choice of the normalization makes subsequent formulas simpler.

Let us mention a few facts about these functions, leaving the details for \S \ref{s.Haar_Hilbert}.
Haar functions, together with the constant function $1$, form an orthogonal set with respect to a natural inner product.
An arbitrary continuous function $f \in C^0(\Omega)$ can be uniquely represented as a \emph{Haar series}:
$$
f = c(f) + \sum_{\omega \in \Omega^*} c_\omega(f) h_\omega \, ,
$$
where $c(f)$, $c_\omega(f)$ are called the \emph{Haar coefficients} of $f$.

\smallskip

A \emph{gauge} is a family  $\bb = (b_\omega)_{\omega \in \Omega^*}$
of positive numbers indexed by words.
Suppose $\ba = (a_n)$ is an evanescent sequence in the sense \eqref{e.eva}.
We say that a gauge $\bb = (b_\omega)$ is \emph{admissible} with respect to $\ba$ if
\begin{equation}\label{e.adm}
\bar{b}_n  = o(a_n)
\quad \text{and} \quad
\log (a_n / \underline{b}_n) = O(n) \, .
\end{equation}
where
$$
\bar{b}_n \coloneqq \max_{|\omega| = n} b_\omega
\quad \text{and} \quad
\underline{b}_n \coloneqq \min_{|\omega| = n} b_\omega \, .
$$
In other words, $b_\omega$ becomes much (but not superexponentially) smaller than $a_n$
as $n = |\omega| \to \infty$.
Examples of admissible gauges are
$b_\omega = n^{-1} a_n$ and $b_\omega = 2^{-n} a_n$.

\begin{lemma}\label{l.intro}
Suppose that $\ba = (a_n)$ is an evanescent sequence,
and that $\bb = (b_\omega)$ is an admissible gauge with respect to $\ba$.
Given numbers $c_\omega$ satisfying $|c_\omega| \le b_\omega$ for all $\omega \in \Omega^*$,
the Haar series $\sum_{\omega \in \Omega^*} c_\omega h_\omega$
represents a unique function in $C^\ba(\Omega)$.
Moreover, the functions obtained in this way form a compact subset $\cH_\bb$ of $C^\ba(\Omega)$.
\end{lemma}

The set of functions $\cH_\bb$ as in the lemma above is called the \emph{Hilbert brick} with gauge $\bb$.
By means of the Haar coefficients, it can be identified with the product space $\prod_{\omega \in \Omega^*} [-b_\omega, b_\omega]$.
We endow each interval $[-b_\omega, b_\omega]$ with the normalized Lebesgue measure,
take the product probability, and using the identification obtain a probability measure
$\P_\bb$ on the Banach space $\cC^\ba(\Omega)$ supported on the compact subset $\cH_\bb$.

The  main technical result of this paper is the following:

\begin{theorem}\label{t.brick}
Suppose that $\ba$ is an evanescent sequence
and that $\bb$ is an admissible gauge with respect to $\ba$.
Then for any $f_0 \in C^\ba(\Omega)$,
$$
\P_\bb \big( \left\{ g \in \cH_\bb \st f_0 + g \in \lock \right\} \big) = 1 \, .
$$
\end{theorem}

The conclusion of the theorem is that
the measure $\P_\bb$ is transverse of the complement of the set $\lock$.
In particular $\lock$ is prevalent.
So Theorem~\ref{t.brick} implies Theorem~\ref{t.prevalence}.

\begin{remark}\label{r.decay}
Based on experimental evidence, Hunt and Ott \cite{HO} also conjecture that
for typical parameterized families of functions, the Lebesgue measure of the parameters
corresponding to maximizing measures with period $p$ or greater decreases exponentially
with $p$.
Using our methods it is straightforward to construct
spaces of Lipschitz functions and
measures on Hilbert bricks
(which can be regarded as families on infinitely many parameters)
such that the measure of the parameters
corresponding to maximizing measures with period $p$ or greater
decays faster than a given function of $p$.
Such result does not seem especially interesting, so we will not provide details.
\end{remark}

\subsection{Main ideas of the proof of Theorem~\ref{t.brick} and organization of the paper}

Due to high regularity, the functions
considered in this paper are well-approximated by step functions.
Equivalently, their Haar coefficients decay fast. The strategy of
the proof is to show that with high probability the maximizing measure
for the function $f$ is the same as for a step function obtained by
a truncation of the Haar series.

\smallskip

The paper is organized as follows: In Section~\ref{s.Haar_Hilbert}
we prove a few simple facts about Haar series and the compactness of
Hilbert bricks (Lemma~\ref{l.intro}). In Section~\ref{s.rot} we
study in detail some finite-dimensional projections of the set of
invariant measures (rotation sets). We describe these sets in
graph-theoretical terms. This information is basic to the arguments
that follow. In Section~\ref{s.gap} we obtain a sufficient
\emph{gap condition} for the locking property;
it is basically a transversality condition in terms of finitely many Haar coefficients.
In Section~\ref{s.proof} we use evanescence
\eqref{e.eva} and admissibility \eqref{e.adm} to show that with
probability $1$ the transversality condition will be satisfied at
some level, thus proving Theorem~\ref{t.brick}. The final
Section~\ref{s.questions} contains a few open questions and discusses
the obstacles to improving the evanescence condition.

\section{More about Haar series and Hilbert bricks}\label{s.Haar_Hilbert}

In this section we collect a few facts about Haar series and we prove
Lemma~\ref{l.intro}.

\smallskip

We fix some notation.
Let $x = (x_0, x_1, \dots)$ and $y = (y_0, y_1, \dots)$ be two elements of $\Omega = \{0,1\}^\N$.
Let $x \disagree y$ denote the position of first disagreement between the sequences $x$ and $y$, that is, the least $m \in \N$ such that $x_m \neq y_m$, with the convention $x \disagree x = \infty$.
The following properties hold:
$$
x \disagree y = y \disagree x \quad \text{(symmetry),} \qquad
x \disagree z \ge \min(x \disagree y, y \disagree z) \quad \text{(ultrametric property).}
$$
For each $n \in \N$,
we define the \emph{$n$-th variation} of a function $f \colon \Omega \to \R$ as:
\begin{equation}\label{e.var}
\var_n(f) \coloneqq \sup_{x\disagree  y \ge n} \big|f(x)-f(y) \big| \, .
\end{equation}
If the function $f$ satisfies $\var_n(f) = 0$, i.e., it is constant in each cylinder of level $n$,
then it is called a \emph{step functions of level $n$}.
Such functions form a vector space $\cS_n$ of dimension $2^n$.

Recall the definition \eqref{e.Haar_function} of the Haar functions $h_\omega$, where $\omega \in \Omega^*$.
For each $n \in \N$, the set $\{1\} \cup \{h_\omega \st |\omega|<n\}$ forms
a basis of the vector space $\cS_n$.

\smallskip

Let $\beta$ be the unbiased Bernoulli measure on $\Omega$,
i.e.\ the probability measure that assigns equal weights to all
cylinders of the same level. Let $L^2(\beta)$ denote the Hilbert
space of functions that are square-integrable with respect to the
measure $\beta$. Then the set $\{1\} \cup \{h_\omega \st \omega \in
\Omega^*\}$ is an orthogonal basis of $L^2(\beta)$. Thus every $f
\in L^2(\beta)$ can be represented by a \emph{Haar series}:
$$
f = c(f) + \sum_{\omega \in \Omega^*} c_\omega(f) h_\omega  \qquad \text{(equality in $L^2(\beta)$),}
$$
where the \emph{Haar coefficients} are defined as:
\begin{equation}\label{e.Haar_coefficients}
c(f) \coloneqq \mediumint f \dd{\beta} \, , \qquad
c_\omega(f) \coloneqq 2^{|\omega|+2} \mediumint f \, h_\omega \dd{\beta} \, ,
\end{equation}
The \emph{$n$-th approximation} of $f \in L^2(\beta)$
is defined in any of the following equivalent ways:
\begin{itemize}
\item the projection of $f$ on the subspace $\cS_n$ along its orthogonal complement;
\item the sum of the truncated Haar series:
\begin{equation}\label{e.A_truncated}
A_n f \coloneqq c(f) + \sum_{|\omega|<n} c_\omega(f) h_\omega \, ;
\end{equation}
\item the function obtained by averaging $f$ on cylinders of level $n$:
\begin{equation}\label{e.A_average}
A_n f = \sum_{|\omega| = n} \left( 2^n \mediumint_{[\omega]} f \dd\beta \right) \chi_{\omega} \, .
\end{equation}
\end{itemize}
As a consequence of the last characterization we have, in terms of notation \eqref{e.var},
$$
\|f - A_n f\|_\infty \le \var_n(f) \, ,
$$
which converges to $0$ as $n \to \infty$ if $f$ is continuous.
Therefore every continuous function $f$ can be written as a uniformly convergent series:
$$
f = c(f) + \sum_{n=0}^\infty \left( \sum_{\substack{\omega \in \Omega^* \\ |\omega| = n}} c_\omega(f) h_\omega  \right),
$$
which, with some abuse of language, we also call \emph{Haar series}.

\smallskip

Given a sequence of coefficients, when do they form the Haar coefficients of a continuous function?
Note that, as an immediate consequence of definition \eqref{e.Haar_coefficients},
\begin{equation}\label{e.var_to_Haar}
|c_\omega (f)| \le \var_{|\omega|}(f) \, ;
\end{equation}
therefore a necessary condition for the continuity of a function is that its Haar coefficients tend to $0$.
The next proposition gives a sufficient condition:

\begin{proposition}\label{p.Haar_convergence}
Given a family of numbers $(c_\omega)_{\omega\in \Omega^*}$, denote
$\bar{c}_n \coloneqq \sup_{|\omega|=n} |c_\omega|$.
If $\sum \bar{c}_n < \infty$ then
the Haar series
$$
c + \sum_{n=0}^\infty \left( \sum_{|\omega| = n} c_\omega h_\omega  \right),
$$
is normally convergent and defines a continuous function $f$ such that
\begin{align}
\| f \|_\infty &\le |c| + \frac{1}{2}\sum_n \bar{c}_n \, ,
\label{e.Haar_to_sup_norm}
\\
\var_n(f) &\le \sum_{k\ge n} \bar{c}_k \quad \forall n \in \N.
\label{e.Haar_to_var}
\end{align}
\end{proposition}

\begin{proof}
Note that if $\omega$ and $\omega'$ are two different
words of same length then the corresponding Haar functions
$h_\omega$, $h_{\omega'}$ have disjoint support. Using this fact,
the normal convergence claim and estimate \eqref{e.Haar_to_sup_norm}
follow. Let us proceed to prove \eqref{e.Haar_to_var}. Fix distinct
points $x$, $y \in \Omega$ and estimate
$$
|f(x) - f(y)| \le \sum_{k=0}^\infty \bar{c}_k \sum_{|\omega|=k}  |h_\omega(x)-h_\omega(y)| \, .
$$
Let $n \coloneqq x \disagree y$ (the position of first disagreement between $x$ and $y$).
Note that:
\begin{itemize}
    \item if $k < n$ then all terms in the inner sum are $0$;
    \item if $k = n$ then the inner sum contains exactly one nonzero term, namely the one corresponding to $\omega = x_0 \cdots x_{n-1} = y_0 \dots y_{n-1}$ (which by convention is the empty word when $n=0$), whose value is $1$;
    \item if $k > n$ then the inner sum contains exactly two nonzero terms, namely the ones corresponding to $\omega = x_0 \cdots x_{n-1}$ and $\omega = y_0 \dots y_{n-1}$, both of which have value  $1/2$.
\end{itemize}
Therefore we obtain that $|f(x) - f(y)| \le \sum_{k \ge n} \bar{c}_k$,
completing the proof of \eqref{e.Haar_to_var}.
\end{proof}

\smallskip

Now let us specialize to the Banach spaces $C^\ba(\Omega)$ considered on \S\ref{ss.precise}.
Fix a sequence $\ba = (a_n) \searrow 0$.
Note that the least Lipchitz constant $\Lip_\ba(\mathord{\cdot})$ with respect to the metric $d_\ba$
can be expressed in terms of the variations~\eqref{e.var} as
\begin{equation}\label{e.var_to_Lip}
\Lip_\ba (f) = \sup_{n \in \N} \frac{\var_n(f)}{a_n} \, .
\end{equation}
By \eqref{e.var_to_Haar},
the Haar coefficients of any $f \in C^\ba(\Omega)$ satisfy the bounds:
\begin{equation}\label{e.Lip_to_Haar}
|c_\omega(f)| \le a_n \Lip_\ba(f) \quad \text{for all $\omega \in \Omega^*$ with $|\omega|=n$.}
\end{equation}
Conversely, if the coefficients $(c_\omega)$ in Proposition~\ref{p.Haar_convergence}
satisfy $\sum_{k \ge n} \bar{c}_k = O(a_n)$ then they are the Haar coefficients of
a function $f \in C^\ba(\Omega)$.

\smallskip

Given a gauge $\bb = (b_\omega)_{\omega \in \Omega^*}$
(i.e., a family of positive numbers indexed by words) such that
\begin{equation}\label{e.b_conditions}
\sum \bar{b}_n < \infty \, , \quad \text{where} \quad
\bar{b}_n \coloneqq \sup_{|\omega| = n} b_\omega \, ,
\end{equation}
the \emph{Hilbert brick} with gauge $\bb$ is:
\begin{equation}
\cH_\bb \coloneqq \left\{ \sum_{\omega \in \Omega^*} c_\omega h_\omega \st |c_\omega| \le b_\omega \right\} \, .
\end{equation}
Then $\cH_\bb$ is a well-defined set of continuous functions,
as an immediate application of Proposition~\ref{p.Haar_convergence}.
By the observations above, if
$\sum_{k \ge n} \bar{b}_k = O(a_n)$ then
$\cH_\bb \subset C^\ba(\Omega)$.
We have the following compactness criterion:

\begin{proposition}\label{p.compact}
If $\sum_{k \ge n} \bar{b}_k = o(a_n)$ then
the Hilbert brick $\cH_\bb$ is a compact subset of the Banach space $C^\ba(\Omega)$.
\end{proposition}

\begin{proof}
The product topology on $Q \coloneqq \prod_{\omega \in \Omega^*} [-b_\omega,b_\omega]$ makes it a compact space.
Using properties \eqref{e.Haar_to_sup_norm}, \eqref{e.Haar_to_var}, and \eqref{e.var_to_Lip}, together with the assumption on the series $\sum \bar{b}_k$, one checks that the map
$$
(c_\omega)_{\omega\in\Omega^*} \in Q \mapsto \sum_{\omega \in \Omega^*} c_\omega h_\omega \in C^\ba(\Omega)
$$
is continuous.
Therefore its image $\cH_\bb$ is compact.
\end{proof}

\begin{proof}[Proof of Lemma~\ref{l.intro}]
If $\ba = (a_n)$ is an evanescent sequence then $\sum_{k \ge n} a_k = O(a_n)$.
So if $\bb = (b_\omega)$ is an admissible gauge with respect to $\ba$
then $\sum_{k \ge n} \bar{b}_k = o(a_n)$.
Applying Proposition~\ref{p.compact} we conclude that $\cH_\bb$
is a compact subset of $C^\ba(\Omega)$.
\end{proof}

As the reader will have noticed, the second admissibility condition
in \eqref{e.adm} is not needed for the validity of
Lemma~\ref{l.intro}. It will be fundamental for the proof of
Theorem~\ref{t.brick}, however.

\section{Finite-dimensional ergodic optimization}\label{s.rot}


\subsection{Rotation sets}

As in \S\ref{ss.precise}, consider a general continuous transformation $T$
of a compact metric space $X$.
The set $\cM_T$ of invariant probabilities
is a compact convex subset of the space $\cM_{\pm}(X)$
of all finite signed Borel measures, endowed with the usual weak topology.

Let $F \subset C^0(X)$ be a finite-dimensional space of continuous functions.
Given a basis $\{f_1, \dots, f_n\}$ of $F$,
consider the  continuous linear map
\begin{equation}\label{e.abstract_projection}
\pi \colon \cM_\pm(X) \to \R^n
\quad \text{defined by} \quad
\pi(\mu) = \big( \langle f_1, \mu\rangle, \dots, \langle f_n, \mu\rangle \big) \, .
\end{equation}
A change of basis produces a conjugate map $\pi$.
The image of $\cM_T$ under $\pi$ is a convex compact subset of $\R^n$,
called a \emph{generalized rotation set}, or simply a \emph{rotation set}.
The terminology is not especially good, but since it seems to be established in the literature
\cite{Ziemian,Jenkinson_rotation,KW} we will conform with it.
The prime example is when $(f_1,\dots,f_n)$ is the displacement vector of a map of the $n$-torus homotopic to the identity.
An example which is much more related to the theme of this paper is
the Bousch's ``fish'' \cite{Bousch_poisson}, where the dynamics is the doubling map
$T(\theta) = 2\theta$ on the circle $\R/2\pi\Z$
and $F$ is spanned by the functions $\cos \theta$ and $\sin \theta$,
so the projection $\pi$ consists on computing two Fourier coefficients.

It is shown in \cite{KW} that arbitrary compact convex sets
in $\R^n$ appear as rotation sets of shift transformations.
On the other hand, a situation where rotation sets are relatively simple was uncovered by Ziemian~\cite{Ziemian}: if $T$ is a transitive subshift of finite type
and $F$ is composed by step functions constant on cylinders of level $2$
then the associated rotation set is a polytope.

\smallskip

Let us say that a rotation set $R = \pi(\cM_T)$  has the \emph{injectivity property}
if for each extreme point $\rho$ of $R$
there exists an unique $\mu \in \cM_T$ such that $\pi(\mu) = \rho$.
Examples are given by Ziemian's polytopes
(where the preimages of extreme points are certain periodic measures)
and Bousch's fish
(where the preimages are the Sturmian measures).

\smallskip

Let us discuss the relevance of rotation sets in ergodic optimization.
The map $\pi$ has the following obvious \emph{projection property}:
for all $\mu$, $\tilde{\mu} \in \cM_{\pm}(X)$,
\begin{equation}\label{e.projection_property}
\pi(\mu) = \pi(\tilde\mu) \quad \Leftrightarrow \quad
\langle f , \mu \rangle = \langle f , \tilde{\mu} \rangle
\text{ for all } f \in F \, .
\end{equation}
Moreover, $\pi$ is onto $\R^n$ (since the functions $f_1$, \dots, $f_n$ are linearly independent).
Using these facts,
we can define a bilinear map
$\llangle \mathord{\cdot} , \mathord{\cdot} \rrangle \colon F \times \R^n \to \R$,
by
\begin{equation}\label{e.new_inner_product}
\llangle f , v \rrangle \coloneqq \langle f , \mu \rangle
\quad \text{where $\mu$ is such that } \pi(\mu) = v.
\end{equation}
In particular, the ergodic supremum \eqref{e.ergsup}
of a function $f \in F$ can be expressed as:
\begin{equation}\label{e.finite_dim_opt}
\ergsup f =
\sup_{v \in R} \llangle f , v \rrangle \, .
\end{equation}
Therefore $R$ and $\llangle \mathord{\cdot} , \mathord{\cdot}
\rrangle$ encode basically all the information necessary to do
ergodic optimization for functions in $F$, which becomes a
finite-dimensional problem. In the particular case that $R$ is a
polytope, the computation of $\ergsup f$ becomes a \emph{linear
programming} problem. If the supremum in \eqref{e.finite_dim_opt} is
attained at a unique point, which is often the case, and $R$ happens
to have the injectivity property, then the original ergodic
optimization problem \eqref{e.ergsup} has a unique solution,
i.e., the function $f \in F$ has a unique maximizing measure.

\subsection{A family of rotation sets for the shift}\label{s.rot_for_shift}

Let us focus on the shift transformation $\sigma$ on $\Omega = \{0,1\}^\N$.
For each $n \in \N$, consider the vector space $\cS_n$ of step functions of level $n$,
which has dimension $2^n$.
The characteristic functions of cylinders of level $n$, ordered lexicographically,
form a basis of $\cS_n$.
Let $\pi_n \colon \cM_{\pm}(\Omega) \to \R^{2^n}$ be the corresponding projection map,
as defined in the previous subsection, that is,
\begin{align*}
\pi_0(\mu) &\coloneqq \big( \mu(\Omega) \big) , \\
\pi_1(\mu) &\coloneqq \big( \mu([0]), \mu([1]) \big) , \\
\pi_2(\mu) &\coloneqq \big( \mu([00]), \mu([01]), \mu([10]), \mu([11]) \big) , \text{ etc.}
\end{align*}
If $\mu$ is a probability measure then $\pi_n(\mu)$ is a subset of the unit simplex $\Delta^{2^n-1}$.
We define a rotation set
$$
R_n \coloneqq \pi_n(\cM_\sigma) \, .
$$

\begin{proposition}\label{p.ziemian}
The rotation sets $R_n$ are polytopes, and satisfy the injectivity property.
\end{proposition}


The proposition (or at least its first part) can be deduced from the aforementioned result of Ziemian~\cite{Ziemian}.
It is also related to work by Fuk\'{s}~\cite{Fuks}.
Since we will need finer information about the sets $R_n$
we will provide a self-contained proof.
As in \cite{Ziemian}, we will use some elementary graph theory.

\subsection{Preliminaries on convexity and graph theory}

Let us review some terminology of convex analysis.
Let $C$ be a closed convex subset of Euclidian space.
A subset $\tilde{C} \subset C$ is called a \emph{face} of $C$ if any segment in $C$ whose relative interior intersects $\tilde C$ is actually contained in $\tilde{C}$.
The \emph{proper faces} of $C$ are those different from $\emptyset$ and $C$.
If a singleton $\{v\}$ is a face of $C$ then $v$ is called an \emph{extreme point} of $C$.
If $C$ has finitely many extreme points then they are called \emph{vertices},
and $C$ is called a \emph{(convex) polytope}.
The faces of a polytope are also polytopes; among them,
the $1$-dimensional faces are called \emph{edges}, and proper faces of maximal dimension are called \emph{facets}.

A closed convex set is called a \emph{(convex) polyhedron} if it is an intersection of finitely many affine-half spaces.
A fundamental theorem \cite[Corollary~7.1c]{Schrijver_program} states that bounded polyhedra are polytopes, and vice-versa.

\medskip

Next let us fix some graph theory terminology.
Let $G$ be a directed graph, composed of finitely many \emph{nodes} and \emph{arcs}.\footnote{The terms \emph{vertices} and \emph{edges} are reserved to polytopes.}
We will always assume that no two arcs have the same source and target nodes.
More formally, $G$ is a pair $(N,A)$, where $N$ is a finite set and $A \subset N \times N$. A directed graph $G' = (N',A')$ is a \emph{subgraph} of $G$ if $N'\subset N$ and $A'\subset A$.

A \emph{walk} in $G$ is a finite sequence of the form $n_0 a_1 n_1 \cdots n_{\ell-1} a_\ell n_\ell$
where each $a_i$ is the arc from the node $n_{i-1}$ to the node $n_i$. The positive integer $\ell$ is called the \emph{length} of the walk. A \emph{closed walk} is an equivalence class of walks with equal ending and starting nodes, where ${n_0 a_1 n_1 \cdots n_{\ell-1} a_\ell n_0}$ is equivalent to ${n_1 a_2 n_2 \cdots n_{\ell-1} a_\ell n_0 a_1 n_1}$. A closed walk with no repeated nodes (and consequently, no repeated arcs) is called a \emph{cycle}.

A \emph{circulation} on $G$ is a function that attributes a weight to each arc of $G$
in such a way that:
\begin{itemize}
\item
weights are nonnegative;
\item
the total weight is $1$;
\item
the total weight entering any node is equal to the total weight leaving that node
(Kirchhoff's current law).
\end{itemize}
The \emph{support} of a circulation is the least subgraph where all weights are positive.


The set of functions on the arcs of $G$ can be identified with some $\R^d$.
One such function is a circulation if it satisfies certain linear equalities and inequalities;
therefore the set of circulations is a polyhedron.
Moreover, it is obviously bounded, so it is a polytope.
It is called the \emph{circulation polytope} of $G$.

\begin{proposition}\label{p.faces}
Let $G$ be a directed graph, and let $C$ be its circulation polytope.
Then the faces of $C$ are the circulation polytopes of the subgraphs of $G$,
and the vertices of $C$ are the circulations supported on the cycles of $G$.
\end{proposition}

As a corollary, every circulation on $G$ is a convex combination of circulations supported on cycles.

\begin{proof}
If a polytope is expressed in terms of linear equalities and inequalities, then its faces can be obtained by setting some of the inequalities to equalities: see \cite[p.~101]{Schrijver_program}. 
The inequalities in our case are ``weights are nonnegative''.
Forcing some of the arcs to have zero weights gives rise to circulations in a subgraph $G'$ with the same set of nodes as $G$.
Removing unconnected nodes from a graph does not affect the possible circulations on it.
Therefore for every subgraph $G'$ of $G$, its circulation polytope $C'$ is a face of $C$, and all faces appear in this way.
Note that $C' \neq \emptyset$ if and only if $G'$ contains a cycle.
Moreover, each cycle supports a unique circulation, and different cycles support different circulations,
so $C'$ is $0$-dimensional if and only if $G'$ contains a single cycle.
\end{proof}

\subsection{Graphical itineraries}\label{ss.itineraries}

For each integer $n \ge 1$, the \emph{de~Bruijn graph} $G_n$
is the directed graph defined as follows:
\begin{itemize}
\item the nodes of $G_n$ are exactly the words of length $n-1$ in the alphabet $\{0,1\}$;
\item the arcs of $G_n$ are exactly the words of length $n$ in the alphabet $\{0,1\}$;
\item for each word $\omega$ of length $n$, the source node and the target node of the arc $\omega$
are respectively its initial and final subwords of length $n-1$.
\end{itemize}
These graphs (together with their analogues for larger alphabets)
were introduced independently by de~Bruijn~\cite{deBruijn} and Good~\cite{Good},
and are relevant to many problems in combinatorics and applied mathematics: see \cite{Maurer} and references therein.
The first four de~Bruijn graphs are pictured in Fig.~\ref{f.graphs}.


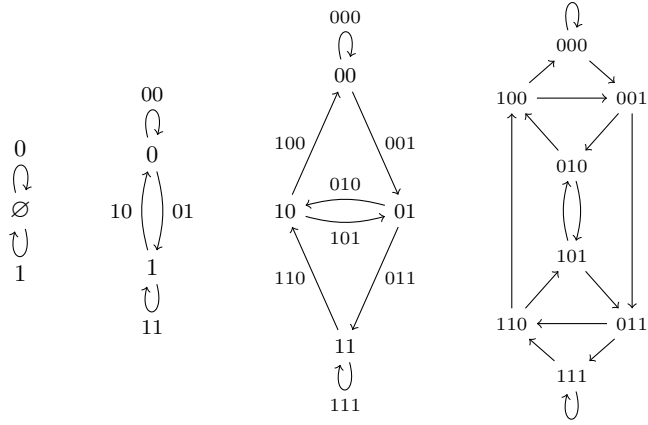
\begin{figure}[htb]
\begin{tikzpicture}[baseline] 
   \node(emptyword) at (0, 0){$\emptyset$};
   \draw(emptyword)edge[loop above]  node{\small $0$} (emptyword);
   \draw(emptyword)edge[loop below] node{\small $1$} (emptyword);
\end{tikzpicture}
\qquad
\begin{tikzpicture}[baseline] 
   \node(0) at (0,.75){\small $0$};
   \node(1) at (0,-.75){\small $1$};
   \draw(0)edge[loop above] node{\footnotesize $00$} (0);
   \draw(0)edge[->,bend left=15] node[right]{\footnotesize $01$} (1);
   \draw(1)edge[loop below] node{\footnotesize $11$} (1);
   \draw(1)edge[->,bend left=15] node[left]{\footnotesize $10$} (0);
\end{tikzpicture}
\qquad
\begin{tikzpicture}[baseline] 
   \node(00) at (  0, 1.8){\footnotesize $00$};
   \node(01) at ( .8,   0){\footnotesize $01$};
   \node(10) at (-.8,   0){\footnotesize $10$};
   \node(11) at (  0,-1.8){\footnotesize $11$};
   \draw(00)edge[loop above]       node{\scriptsize $000$} (00);
   \draw(00)edge[->]               node[right]{\scriptsize $001$} (01);
   \draw(01)edge[->,bend right=15] node[above]{\scriptsize $010$} (10);
   \draw(01)edge[->]               node[right]{\scriptsize $011$} (11);
   \draw(10)edge[->]               node[left] {\scriptsize $100$} (00);
   \draw(10)edge[->,bend right=15] node[below]{\scriptsize $101$} (01);
   \draw(11)edge[->]               node[left] {\scriptsize $110$} (10);
   \draw(11)edge[loop below]       node{\scriptsize $111$} (11);
\end{tikzpicture}
\qquad
\begin{tikzpicture}[baseline] 
   \node(000) at (  0, 2.2){\scriptsize $000$};
   \node(001) at ( .8, 1.5){\scriptsize $001$};
   \node(100) at (-.8, 1.5){\scriptsize $100$};
   \node(010) at (  0,  .6){\scriptsize $010$};
   \node(101) at (  0, -.6){\scriptsize $101$};
   \node(011) at ( .8,-1.5){\scriptsize $011$};
   \node(110) at (-.8,-1.5){\scriptsize $110$};
   \node(111) at (  0,-2.2){\scriptsize $111$};
   \draw(000)edge[loop above]      (000);
   \draw(000)edge[->]              (001);
   \draw(001)edge[->]              (010);
   \draw(001)edge[->]              (011);
   \draw(010)edge[->]              (100);
   \draw(010)edge[->,bend left=15] (101);
   \draw(011)edge[->]              (110);
   \draw(011)edge[->]              (111);
   \draw(100)edge[->]              (000);
   \draw(100)edge[->]              (001);
   \draw(101)edge[->,bend left=15] (010);
   \draw(101)edge[->]              (011);
   \draw(110)edge[->]              (100);
   \draw(110)edge[->]              (101);
   \draw(111)edge[->]              (110);
   \draw(111)edge[loop below]      (111);
\end{tikzpicture}
\caption{The de~Bruijn graphs $G_n$ for $1 \le n \le 4$.}\label{f.graphs}
\end{figure}

Any orbit of the one-sided shift $\sigma \colon \Omega \to \Omega$ induces in the obvious way a semi-infinite walk on the de~Bruijn graph $G_n$, called the \emph{$G_n$-itinerary} of the orbit.
If the orbit is periodic then we obtain a closed walk of length equal to the period.
If $\mu$ is the measure supported on the periodic orbit, we denote the induced closed walk
on $G_n$ by $W_n(\mu)$.
The following sets of measures will play an important role:
\begin{equation*}
\cC_n \coloneqq
\big\{ \mu\in \cM_\sigma \text{ periodic} \st W_n(\mu) \text{ is a cycle in } G_n \big\} \, ,
\end{equation*}
or equivalently
\begin{equation}\label{e.ifyoudontlikegraphs}
\begin{aligned}
\cC_n = \big\{ \mu \in \cM_\sigma \text{ periodic} \st \text{each cylinder }&\text{of level $n-1$ contains} \\
&\text{at most one point of $\supp \mu$}\big\} \, .
\end{aligned}
\end{equation}
There is an obvious bijection between $\cC_n$ and the set of cycles in $G_n$.
The initial terms of the sequence $(\cC_n)$ are:
$$
\cC_1 = \big\{ \overline{0}, \, \overline{1} \big\} \subset
\cC_2 = \big\{ \overline{0}, \, \overline{1}, \, \overline{01} \big\} \subset
\cC_3 = \big\{ \overline{0}, \, \overline{1}, \, \overline{01}, \, \overline{001}, \, \overline{011}, \, \overline{0011} \big\} \, ,
$$
where $\overline{01}$, for example, denotes the periodic measure supported on the orbit of the point $(0,1,0,1,\dots) \in \Omega$.

\begin{remark}\label{r.complexity}
The nested union $\bigcup_{n\ge 1} \cC_n$ equals the set of periodic measures in $\cM_\sigma$.
Following Maurer~\cite{Maurer}, we define the \emph{recursive complexity}
of a periodic measure $\mu$ as the least $n$ such that $\mu \in \cC_n$.\footnote{Note that Maurer's definition is slightly different since his $G_n$ is our $G_{n+1}$.}
Recursive complexity is not monotonic with respect to the period:
for example $\overline{000111}$ and $\overline{01011}$
have recursive complexities $4$ and $5$, respectively.
\end{remark}

\begin{remark}\label{r.hamiltonian}
It seems to be a difficult task to count precisely the number of elements of $\cC_n$
(or equivalently the number of cycles in $G_n$): see \cite{Maurer} for related results.
Let us mention a celebrated result by de~Bruijn \cite{deBruijn}:
for $n \ge 2$, there are exactly
$2^{2^{n-2} - n + 1}$
cycles in $G_n$ of maximal length $2^{n-1}$;
these are the Hamiltonian cycles of $G_n$, i.e., the cycles that visit every node.\footnote{Incidentally, these cycles are exactly the lifts of the Eulerian closed walks on $G_{n-1}$.}
\end{remark}

\subsection{The structure of the rotation sets $R_n$}\label{ss.structure}

Let us use the above material to describe precisely the rotation sets
$R_n =  \pi_n(\cM_\sigma)$ introduced in \S\ref{s.rot_for_shift}.
The following result is a finer version of Proposition~\ref{p.ziemian}:

\begin{proposition}\label{p.rot_and_dbg}
For all $n \ge 1$, the rotation set $R_n$ is isomorphic to
the circulation polytope of the de~Bruijn graph $G_n$.
Moreover, $R_n$ has the injectivity property:
the restriction of $\pi_n$ to the set $\cC_n$ is a bijection onto the vertices of $R_n$.
\end{proposition}

\begin{proof}
Suppose $\rho \in R_n$, that is, $x = \pi_n(\mu)$ for some $\mu \in \cM_\sigma$.
The vector $x$ gives a weight to each cylinders of level $n$
and thus to each  arc of $G_n$.
We claim that this weight function is a circulation.
Weights are obviously nonnegative and sum to $1$,
so we are left to verify Kirchhoff's current law.
Given any node $\omega$, that is, a word of length $n-1$, we have
the following disjoint unions:
$$
[\omega] = [\omega 0] \cup [\omega 1]  \, , \qquad
\sigma^{-1}([\omega]) = [0 \omega] \cup [1 \omega]  \, .
$$
So $\mu([\omega])$ is the total weight leaving the node $\omega$,
while $\mu(\sigma^{-1}([\omega]))$ is the total weight entering it.
They are equal by invariance of the measure.
So Kirchhoff's current law is satisfied,
and we have shown that $x$ induces a circulation on $G_n$.

Conversely, let us check that every circulation on $G_n$ is of the form above.
This is obvious if the circulation is supported on a cycle,
and since a general circulation is a convex combination of circulations supported on cycles
(by Proposition~\ref{p.faces}),
the claim follows.

The injectivity property is clear, and the proposition is proved.
\end{proof}

Proposition~\ref{p.rot_and_dbg} permits us to deduce other properties of rotation set $R_n$.
For example, the dimension of $R_n$ is $2^{n-1}$.\footnote{Sketch of proof: Consider circulations of $G_n$ as vectors in $\R^{2^n}$. Then $R_n = \Delta \cap K$, where $\Delta$ is the standard unit simplex in $\R^{2^n}$, and  $K$ is the vector subspace corresponding to Kirchhoff's current law. One checks that $K$ has codimension $2^{n-1} - 1$ and is transverse to $\Delta$. It follows that $\dim R_n = 2^{n-1}$.}
The face structure of $R_n$ can be studied using Proposition~\ref{p.faces}:
an example is shown in Fig.~\ref{f.bird}.
Since $G_n$ has $2^n$ arcs, it follows from Propositions~\ref{p.rot_and_dbg}
and \ref{p.faces} that $R_n$ has no more than $2^{2^n}$ faces.

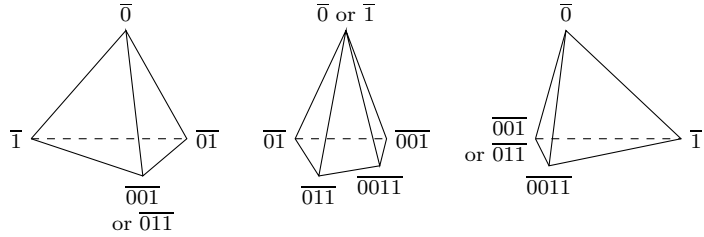
\begin{figure}[htb]
\begin{tikzpicture}[x=.45mm,y=.45mm,font=\footnotesize]
   \coordinate[label=above:$\overline{0}$]     (L0)    at ( 76+12,124);
   \coordinate[label=right:$\overline{01}$]    (L01)   at ( 94+12, 92);
   \coordinate                                 (L001)  at ( 81+12, 81);
   \coordinate[label=left:$\overline{1}$]      (L1)    at ( 48+12, 92);
   \node[below,align=center] at (L001) {$\overline{001}$\\or $\overline{011}$};
   \draw(L1)--(L0);
   \draw[dashed](L1)--(L01);
   \draw(L1)--(L001);
   \draw(L0)--(L01)--(L001)--cycle;
   \coordinate[label=above:$\overline{0}$ or $\overline{1}$]   (M0)    at (153,124);
   \coordinate[label=left:$\overline{01}$]                     (M01)   at (138, 92);
   \coordinate[label=below:$\overline{011}$]                   (M011)  at (145, 81);
   \coordinate[label=below:$\overline{0011}$]                  (M0011) at (163, 84);
   \coordinate[label=right:$\overline{001}$]                   (M001)  at (165, 92);
   \draw(M0)--(M01)--(M011)--cycle;
   \draw(M0)--(M0011)--(M001)--cycle;
   \draw[dashed](M01)--(M001);
   \draw(M011)--(M0011);
   \coordinate[label=above:$\overline{0}$]     (R0)    at (218,124);
   \coordinate                                 (R001)  at (209, 92);
   \coordinate[label=below:$\overline{0011}$]  (R0011) at (213, 84);
   \coordinate[label=right:$\overline{1}$]     (R1)    at (252, 92);
   \node[left,align=right] at (R001) {$\overline{001}$\\or $\overline{011}$};
   \draw(R001)--(R0);
   \draw[dashed](R001)--(R1);
   \draw(R001)--(R0011);
   \draw(R0)--(R0011)--(R1)--cycle;
\end{tikzpicture}
\caption{The polytope $R_3$ has $6$ facets: $4$ tetrahedra and $2$ pyramids (pictured). The total number of faces is $1+6+13+13+6+1=40$.}\label{f.bird}
\end{figure}

\smallskip

We close this section mentioning other facts that,
despite not being needed for the rest of the paper,
have independent interest.

\begin{remark}
In view of the characterization \eqref{e.finite_dim_opt},
the computation of the ergodic supremum of a function in $\cS_n$ becomes
equivalent to a \emph{maximum cycle mean problem} for the de~Bruijn graph $G_n$.
These problems have been studied in the combinatorics optimization literature,
and efficient algorithms are known: see \cite{Karp,DIG}.\footnote{It is interesting to note that Karp's algorithm relies on a graph-theoretical version of Ma\~n\'e lemma. In this regard, a function $f$ on the set of arcs is called a \emph{coboundary} if there exists a function $h$ on its set of nodes such that
$f(a) = h(n_1) - h(n_0)$ whenever $a$ is an arc from node $n_0$ to node $n_1$.}
Graphs with \emph{random} weights on the arcs had also been studied: see \cite{MW}.
\end{remark}

\begin{remark}\label{r.proj_lim}
The Poulsen simplex $\cM_\sigma$ can be recovered from the sequence of polytopes $(R_n)$
as an inverse limit.
\end{remark}

\begin{remark}
Using the appropriate graphs,
it is straightforward to adapt the results above to general subshifts of finite type
and in particular reobtain the results of~\cite{Ziemian}.
\end{remark}

\section{A gap criterion for the locking property}\label{s.gap}

\subsection{Statement of the criterion}

Consider a step function  $g \in \cS_n$ of level $n \ge 1$.
By the facts discussed in the previous
section, in terms of notation \eqref{e.new_inner_product} we have
$\ergsup g = \llangle g , v \rrangle$ for some vertex $v$ of the
polytope $R_n$.
Define the \emph{$n$-gap} of $g$ as (see Fig.~\ref{f.gap}):
$$
\gap_n (g) \coloneqq \llangle g, v \rrangle - \max_{\substack{w \text{ vertex of } R_n \\ w\neq v}}  \llangle g, w \rrangle \, .
$$
Note that in the expression above it is sufficient to consider the vertices $w$
such that $[v,w]$ is an edge of the polytope $R_n$.

Let $\mu$ be the measure in $\cM_\sigma$ such that $v = \pi_n(\mu)$.
Then $\ergsup g = \langle g , \mu \rangle$, i.e.\ $\mu$ is a maximizing measure for $g$;
moreover, $\mu \in \cC_n$, and in particular it is a periodic measure.
The $n$-gap can be written as:
$$
\gap_n (g) = \langle g, \mu \rangle - \max_{\nu \in \cC_n \setminus \{\mu\} }  \langle g, \nu \rangle \, .
$$
Note that $\gap_n (g)$ is positive if $g$ has no other maximizing measure in $\cM_\sigma$,
and is zero otherwise.
In the former case, every $\tilde g \in \cS_n$ sufficiently close to $g$
also has an unique maximizing measure, which is $\mu$.

\begin{figure}[htb]
\begin{tikzpicture}[scale=.35, rotate=-15, font=\small]
   \coordinate (A) at (3.5,6);
   \coordinate (B) at ( .8,5);
   \coordinate (C) at (0  ,3);
   \coordinate (D) at (1  ,1);
   \coordinate (E) at (3  ,0);
   \coordinate (F) at (6  ,2);
   \coordinate (G) at (5.5,4);
   \draw[fill=black!20!white](A)--(B)--(C)--(D)--(E)--(F)--(G)--cycle;
   \foreach \point in {A,B,C,D,E,F,G} \fill [black] (\point) circle (3pt);
   \foreach \y in {6,5,3,1,0,2,4} \draw[very thin] (-.5,\y)--(7,\y);
   \node[above right] at (A) {$v = \pi_n(\mu)$};
   \node at (3.15,3.55) {$R_n$};
   \draw[<->](-.5,6)--(-.5,5);
   \node[left] at (-.5,5.5) {$\gap_n(g)$};
   \draw [decorate,decoration={brace,amplitude=10pt,mirror},xshift=3pt] (7,0) -- (7,6);
   \node[right] at (7.8,3) {level sets of $\llangle g ,  \mathord{\cdot}\rrangle$};
\end{tikzpicture}
\caption{Gap of a function $g \in \cS_n$.} 
\label{f.gap}
\end{figure}

Recall that $A_n f$ denotes the truncated Haar series \eqref{e.A_truncated} of a function $f$, and so it belongs to $\cS_n$.

The following lemma plays a pivotal role in this paper:

\begin{lemma}[Gap criterion]\label{l.gap}
Let $f \in C^0(\Omega)$ and $n \in \N$.
Assume that the following \emph{gap condition} is satisfied:
\begin{equation}\label{e.gap}
\gap_n(A_n f) > \sum_{k=n}^\infty (k-n+1) \max_{|\omega|=k} |c_\omega(f)| \, .
\end{equation}
Then the maximizing measure for $A_n f$ (which is unique and periodic)
is also the unique maximizing measure for $f$.
\end{lemma}

In informal terms, if the tail of the Haar series is small compared to the gap
of its initial part, then it does not influence the maximizing measure;
in particular property $\mathsf{P}$ is satisfied.
The lemma actually permits us to obtain the stronger locking property:

\begin{corollary}[Gap criterion version 2]\label{c.gap}
Consider a sequence $\ba = (a_n) \searrow 0$ such that:
\begin{equation}\label{e.summability}
\sum_{n=0}^\infty n a_n < \infty \, .
\end{equation}
If a function $f \in C^\ba(\Omega)$ satisfies the gap condition \eqref{e.gap}
for some $n \in \N$ then $f \in \lock$.
\end{corollary}

\begin{proof}
In view of estimate \eqref{e.Lip_to_Haar},
assumption \eqref{e.summability} implies that, for every $n \in \N$,
the RHS of \eqref{e.gap} converges for any function
$f \in C^\ba(\Omega)$, and moreover it depends continuously on $f$.
The LHS is also continuous, so the functions $f$ in the space $C^\ba(\Omega)$ that satisfy the gap condition \eqref{e.gap} form an open subset $\cU_n$.
By Lemma~\ref{l.gap}, each $f \in \cU_n$ has a unique maximizing measure,
which coincides with the unique maximizing measure for $A_n f$ and so is periodic.
As observed before, maximizing measures of functions in $\cS_n$
are locally constant whenever the gap is positive.
It follows that $\cU_n \subset \lock$, as we wanted to show.
\end{proof}

The proof of Lemma~\ref{l.gap} will occupy the rest of this section,
and requires a few other lemmas.

\subsection{Basins of repellency}\label{ss.basin}

Given $\mu \in \cM_\sigma$ and $k \in \N$,
let $B_{\mu,k} \subset \Omega$ be the union of all cylinders of level $k$
that intersect the support of $\mu$.

\begin{lemma}\label{l.referee}
If $\mu \in \cC_n$ then for all $s \in \N$ we have:
$$
\bigcap_{j=0}^s \sigma^{-j}(B_{\mu,n}) = B_{\mu,n+s} \, .
$$
\end{lemma}

\begin{proof}
If $x \in B_{\mu,n+s}$ then the cylinder of level $n+s$ that contains $x$ intersects $\supp \mu$ at some point $z$.
So the $n+s$ initial symbols of $x$ and $z$ coincide.
It follows that for each $j \in \{0,\dots,s\}$, the cylinder of level $n$ that contains $\sigma^j(x)$ also contains $\sigma^j(z)$,
showing that $x \in \sigma^{-j}(B_{\mu,n})$.

Conversely, consider a point $y \in \bigcap_{j=0}^s \sigma^{-j}(B_{\mu,n})$.
Then for each $j \in \{0,\dots,s\}$, the cylinder $C_j$ of level $n$ that contains $\sigma^j(y)$ intersects $\supp \mu$ at some point $z^{(j)}$.
If $j<s$ then the set $\sigma(C_j)$, being a cylinder of level $n-1$ that intersects the higher level cylinder $C_{j+1}$, must contain it.
In particular, the cylinder $\sigma(C_j)$ contains both points $z^{(j+1)}$ and $\sigma(z^{(j)})$ from $\supp \mu$.
Using the hypothesis $\mu \in \cC_n$ and recalling \eqref{e.ifyoudontlikegraphs}, we conclude that $z^{(j+1)} = \sigma(z^{(j)})$, for each $j \in \{0,\dots,s-1\}$.
Hence the set $D \coloneqq \bigcap_{j=0}^s \sigma^{-j}(C_j)$, which is a cylinder of level $n+s$ containing $y$, also contains $z^{(0)}$,
and therefore is contained in $B_{\mu,n+s}$.
So $y \in B_{\mu,n+s}$, as we wanted to show.
\end{proof}


As a consequence of the lemma, if $\mu \in \cC_n$ then
$\bigcap_{j=0}^\infty \sigma^{-j}(B_{\mu,n}) = \supp \mu$.
Therefore we say that the set $B_{\mu,n}$ is a \emph{basin of repellency}
for the periodic orbit that supports $\mu$.

Let $\mathsf{c}$ denote the complement operation, that is, $B^\mathsf{c} \coloneqq \Omega \setminus B$.

\begin{lemma}\label{l.cancel}
Given measures $\mu \in \cC_n$ and $\xi \in \cC_k$ with $n \le k$ and $\per(\mu) \le \per(\xi)$,
and a function $g \in \cS_k$, we have
$$
\langle g, \xi \rangle - \langle g, \mu \rangle
\le 2 (k-n+1) \|g\|_\infty  \, \xi(B_{\mu,n}^\mathsf{c}) \, .
$$
\end{lemma}

Note that $n \le k$ does not imply that $\per(\mu) \le \per(\xi)$ in the statement above;
an example appears in Remark~\ref{r.complexity}.

\begin{proof}[Proof of Lemma~\ref{l.cancel}]
Write $m \coloneqq \per(\mu)$ and $\ell \coloneqq \per(\xi)$.
The step function $g \in \cS_k$ induces a function on the arcs of $G_k$,
which by simplicity we denote by the same symbol.
Let $A$ and $B$ be respectively the sets of arcs
in the cycles $W_k(\mu)$ and $W_k(\xi)$ (in the notation from \S\ref{ss.itineraries}).
So $\# A = m$ and $\# B = \ell$.
Let $i \coloneqq \#(A\cap B)$.
Then:
\begin{align*}
\langle g, \xi \rangle - \langle g, \mu \rangle
&= \frac{1}{\ell} \sum_{a \in B} g(a) - \frac{1}{m} \sum_{a \in A} g(a) \\
&= \frac{1}{\ell} \sum_{a \in B \setminus A} g(a) - \frac{1}{m} \sum_{a \in A \setminus B} g(a) +
\left(\frac{1}{\ell} - \frac{1}{m}\right) \sum_{a \in A \cap B} g(a)  \\
&\le \left\{ \frac{\ell-i}{\ell} + \frac{m-i}{m} + \left|\frac{i}{m} - \frac{i}{\ell}\right| \right\} \|g\|_\infty \, .
\end{align*}
By assumption, $m \le \ell$, and so the quantity between curly brackets is $2(\ell-i)/\ell = 2\xi(B_{\mu,k}^\mathsf{c})$.
Using Lemma~\ref{l.referee} and invariance of $\xi$,
$$
\xi(B_{\mu,k}^\mathsf{c}) = \xi \left( \bigcup_{j=0}^{k-n} \sigma^{-j}(B_{\mu,n}^\mathsf{c}) \right)
\le (k-n+1)\xi(B_{\mu,n}^\mathsf{c}) \, .
$$
The lemma follows.
\end{proof}

\subsection{Proof of the gap criterion}

Fix $f \in C^0(\Omega)$ satisfying the gap condition \eqref{e.gap} for some $n \in \N$.
Let $\mu \in \cC_n$ be the maximizing measure for $A_n f$,
and fix an arbitrary $\nu \in \cM_\sigma$.
We need to prove that $\langle f, \mu \rangle \ge \langle f, \nu \rangle$,
with equality only if $\nu = \mu$.

For each integer $k \ge n$, we will deal with the basin of repellency
$B_{k,\mu}$ introduced in \S\ref{ss.basin},
which we will denote by $B_k$ for simplicity.

For each $k \ge n$, consider convex combinations of the form
\begin{equation}\label{e.nu_k_decomp}
\nu_k = \sum_{\xi \in \cC_k} p_k(\xi) \, \xi
\end{equation}
(so the numbers $p_k(\xi)$ are nonnegative and their sum over $\xi$ is $1$)
such that
\begin{equation}\label{e.nu_k_property}
\pi_k(\nu_k) = \pi_k(\nu).
\end{equation}
Such combinations always exist, but uniqueness may fail.
We choose and fix a combination such that the following quantity:
\begin{equation}\label{e.uk}
u_k \coloneqq \sum_{\substack{\xi \in \cC_k \\ \per(\xi) \ge \per(\mu)}} p_k(\xi) \, \xi(B_n)
\end{equation}
is minimal among all possible decompositions; its existence follows from a compactness argument.

For example, if $\nu = \beta$ (the unbiased Bernoulli measure) and $\per(\mu) = 4$ then $u_3 = 0$, because we can take $\nu_3$ as a convex combination of of measures in $\cC_3$
of periods less than $4$, namely
$\nu_3 = \tfrac{1}{8} \, \overline{0} + \tfrac{1}{8} \, \overline{1} + \tfrac{3}{8} \, \overline{001} + \tfrac{3}{8} \, \overline{011}$.

\begin{lemma}\label{l.monotonicity}
The sequence $(u_k)_{k \ge n}$ is monotone nondecreasing.
\end{lemma}

\begin{proof}
Fix $k \ge n$.
For each $\eta \in \cC_{k+1}$, write the circulation $\pi_k(\eta)$ as a convex combination of
vertices of $R_k$:
\begin{equation}\label{e.etas}
\pi_k(\eta) = \sum_{\xi \in \cC_k} q_\eta(\xi) \, \pi_k(\xi) \, .
\end{equation}
Note that
\begin{equation}\label{e.per_ineq}
\per(\xi) > \per(\eta)  \ \Rightarrow \  q_\eta(\xi) = 0 \, ;
\end{equation}
indeed for each $\xi \in \cC_k$ the closed walk $W_k(\xi)$
visits exactly $\per(\xi)$ nodes of $G_k$,
while the closed walk $W_k(\eta)$ visits at most $\per(\eta)$ nodes.

On the other hand, 
$$
\pi_k(\nu) = \pi_k (\pi_{k+1}(\nu)) =
\pi_k \left( \sum_{\eta \in \cC_{k+1}} p_{k+1}(\eta) \, \eta \right) =
\sum_{\eta \in \cC_{k+1}} p_{k+1}(\eta) \, \pi_k(\eta).
$$
Substituting the expressions \eqref{e.etas} and manipulating, we obtain:
$$
\pi_k(\nu) = \pi_k(\tilde \nu), \quad \text{where} \quad
\tilde \nu \coloneqq \sum_{\xi \in \cC_k} \left( \sum_{\eta \in \cC_{k+1}} p_{k+1}(\eta) q_\eta(\xi) \right) \pi_k(\xi) \, ,
$$
which should be compared with \eqref{e.nu_k_decomp} and \eqref{e.nu_k_property}.
Therefore:
\begin{alignat*}{2}
u_k
&\le \sum_{\substack{\xi \in \cC_k \\ \per(\xi) \ge \per(\mu)}}
\left( \sum_{\eta \in \cC_{k+1}} p_{k+1}(\eta) q_\eta(\xi) \right) \xi(B_n)
&\enskip &\text{(by minimality of \eqref{e.uk})} \\
&\le \sum_{\substack{\eta \in \cC_{k+1} \\ \per(\eta) \ge \per(\mu)}}
\sum_{\xi \in \cC_k} p_{k+1}(\eta) q_\eta(\xi) \, \xi(B_n)
&\enskip &\text{(by \eqref{e.per_ineq})}  \\
&= \sum_{\substack{\eta \in \cC_{k+1} \\ \per(\eta) \ge \per(\mu)}}
p_{k+1}(\eta) \left( \sum_{\xi \in \cC_k} q_\eta(\xi) \, \xi \right)(B_n) \\
&= \sum_{\substack{\eta \in \cC_{k+1} \\ \per(\eta) \ge \per(\mu)}}
p_{k+1}(\eta) \, \eta(B_n)
&\enskip &\text{(using \eqref{e.etas}, $\chi_{B_n} \in \cS_n \subset \cS_k$, \eqref{e.projection_property})}  \\
&= u_{k+1} \, ,
\end{alignat*}
thus proving monotonicity.
\end{proof}

We are ready to prove the gap criterion:

\begin{proof}[Proof of Lemma~\ref{l.gap}]
Let $f$, $n$, $\mu$, $\nu$ be as explained above.
Write
$$
f = A_n f + \sum_{k=n}^\infty g_k \, , \quad \text{where} \quad
g_k \coloneqq \sum_{|\omega|=k} c_\omega(f) h_\omega \, .
$$
Then, using the projection property \eqref{e.projection_property}
together with \eqref{e.nu_k_decomp} and \eqref{e.nu_k_property}, we have:
\begin{align}
\langle f, \mu - \nu \rangle
&=
\langle A_n f, \mu - \nu \rangle - \sum_{k=n}^\infty \langle g_k, \nu - \mu \rangle \notag \\
&=
\langle A_n f, \mu - \nu_n \rangle - \sum_{k=n}^\infty \langle g_k, \nu_k - \mu \rangle \notag \\
&= \underbrace{\sum_{\xi \in \cC_n} p_n(\xi) \langle  A_n f, \mu -
\xi \rangle}_{\circled{1}} \ - \ \sum_{k=n}^\infty \underbrace{\sum_{\xi \in
\cC_k} p_k(\xi) \langle g_k, \xi - \mu \rangle}_{\circled{2}_k} \,  .
\label{e.onetwo}
\end{align}
We will estimate the terms $\circled{1}$ and $\circled{2}_k$ from below and from
above, respectively.

On the one hand, by definition of the gap,
for all $\xi \in \cC_n$ different from $\mu$ we have
$\langle A_n f, \mu - \xi \rangle \ge \gap_n(A_n f)$,
and therefore
$$
\circled{1} \ge (1- p_n(\mu))  \gap_n(A_n f)
            \ge (1- u_n)  \gap_n(A_n f)  \, .
$$
On the other hand,
if $\xi \in \cC_k$ has $\per(\xi) \ge \per(\mu)$ then by Lemma~\ref{l.cancel},
$$
\langle g_k, \xi - \mu \rangle \le 2\|g_k\|_\infty \, \xi(B_k^\mathsf{c}) \le 2 (k-n+1)\|g_k\|_\infty \,  \xi(B_n^\mathsf{c}) \, .
$$
In the case that $\per(\xi) < \per(\mu)$ we use the trivial estimate
$$
\langle g_k, \xi - \mu \rangle \le 2\|g_k\|_\infty \le 2 (k-n+1) \|g_k\|_\infty \, .
$$
So we have:
\begin{align*}
\circled{2}_k &\le 2(k-n+1) \|g_k\|_\infty \left\{ \sum_{\substack{\xi \in
\cC_k \\ \per(\xi) <   \per(\mu)}} p_k(\xi)  \ + \sum_{\substack{\xi
\in \cC_k \\ \per(\xi) \ge \per(\mu)}} p_k(\xi) \xi(B_n^\mathsf{c})
\right\} \\
&=   2(k-n+1) \|g_k\|_\infty \, (1-u_k) \\
&\le 2(k-n+1) \|g_k\|_\infty \, (1-u_n),
\end{align*}
where in the last step we used Lemma~\ref{l.monotonicity}.
Substituting these estimates into \eqref{e.onetwo}, we obtain
$$
\langle f, \mu - \nu \rangle \ge \delta (1-u_n) ,
\quad \text{where} \quad
\delta \coloneqq \gap_n(A_n f) - \sum_{k=n}^\infty 2(k-n+1) \|g_k\|_\infty  \, .
$$
Note that $\|g_k\|_\infty = \frac{1}{2} \sup_{|\omega|=k}|c_\omega(f)|$, and so
assumption \eqref{e.gap} means that $\delta>0$.
Moreover, by definition $u_n \le \nu_n(B_n)$, so
$$
\langle f, \mu - \nu \rangle \ge \delta \, \nu_n(B_n^\mathsf{c}) \, .
$$
In particular $\langle f, \mu - \nu \rangle \ge 0$.
Assume that equality holds; then $\nu_n(B_n) = 1$.
Since $B_n$ is a basin of repellency for $\mu$, we have $\nu_n = \mu$.
So $\pi_n(\nu) = \pi_n(\mu)$ and
by the injectivity property of the rotation set $R_n$
we conclude that $\nu = \mu$.
This proves that $\mu$ is the unique maximizing measure for $f$,
as claimed.
\end{proof}

\section{Proof of prevalence}\label{s.proof}

In this section we prove Theorem~\ref{t.brick} which, as explained before,
yields Theorem~\ref{t.prevalence} as a corollary.

A \emph{box} $Q \subset \R^d$ is a product $I_1 \times \cdots \times I_d$
of compact intervals of positive lengths,
the least of which is called the \emph{thickness} of the box and
is denoted by $\tau(Q)$.
Lebesgue measure on $\R^d$ is denoted by $\Leb_d$.
We will estimate the measure of slices of a box, as in Fig.~\ref{f.slice}.

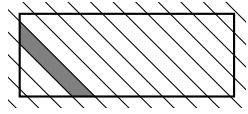
\begin{figure}[htb]
\begin{tikzpicture}
   \clip (-.15,-.15) rectangle (3,1.25);  
   \filldraw[gray] (.7,0)--(1,0)--(0,1)--(0,.7)--cycle; 
   \draw[thick] (0,0) rectangle (2.85,1.1); 
   \foreach \x in {-.2,.1,.4,...,4.1}
       \draw (\x+.5,-.5)--(\x-1.6,1.6); 
\end{tikzpicture}
\caption{Slicing a box.}\label{f.slice}
\end{figure}

If $L$ is a linear functional on $\R^d$, say $L(x) = \sum_{i=1}^d \lambda_i x_i$,
then its \emph{$\infty$-norm} is $\|L\|_\infty \coloneqq \max_i |\lambda_i|$.

\begin{lemma}\label{l.slice}
Let $Q \subset \R^d$ be a box.
Let $\phi \colon \R^d \to \R$ be an affine function with linear part $L \neq 0$.
Then, for all $\delta > 0$,
$$
\frac{\Leb_d(\{x\in Q \st |\phi(x)| \le \delta \})}{\Leb_d(Q)}
\le \frac{2\delta}{\|L\|_\infty \, \tau(Q) } \, .
$$
\end{lemma}

\begin{proof}
Let $Q = I_1 \times \cdots \times I_d$ and $L(x) = \sum_{i=1}^d \lambda_i x_i$
Without loss of generality, assume that $\|L\|_\infty = |\lambda_1|$.
For each $y \in I_2 \times \cdots \times I_d$,
the section $\{x_1 \in I_1 \st |\phi(x_1,y)| \le \delta \}$
has measure at most $2\delta/|\lambda_1|$.
So, by Fubini theorem,
\[
\frac{\Leb_d(\{x\in Q \st |\phi(x)| \le \delta \})}{\Leb_d(Q)}
\le \frac{2\delta}{|\lambda_1|} \,\frac{\Leb_{d-1}(I_2 \times \cdots \times I_d)}{\Leb_d(Q)}
\le \frac{2\delta}{\|L\|_\infty \, \tau(Q)} \, .
\qedhere
\]
\end{proof}

\begin{proof}[Proof of Theorem~\ref{t.brick}]
Fix an evanescent sequence $\ba$, and an admissible gauge $\bb$ with respect to $\ba$.
Fix an arbitrary $f_0 \in C^\ba(\Omega)$.
Recall that our aim to prove that for $\P_\bb$-almost every function
$g$ in the Hilbert brick $\cH_\bb$, the function $f = f_0 + g$
has the locking property.
For that purpose we will use the gap criterion version~2 (Corollary~\ref{c.gap}).
Note that hypothesis~\eqref{e.summability} is automatic from evanescence.

For each $n \in \N$, let
\begin{align*}
\delta_n &\coloneqq \sum_{k=n}^\infty (k-n+1) \big( a_k  \Lip_\ba(f_0) + \overline{b}_k \big)
\qquad \text{and} \\
\cG_n &\coloneqq \big\{ g \in \cH_\bb \st \gap_n \big( A_n(f_0+g) \big) > \delta_n \big\} \, .
\end{align*}
In view of \eqref{e.Lip_to_Haar} and the definition of the Hilbert brick $\cH_\bb$,
the number $\delta_n$ is an upper bound for the RHS of \eqref{e.gap} for
all functions of the form $f = f_0 + g$ with $g \in \cH_\bb$.
So by Corollary~\ref{c.gap} we have:
\begin{equation}\label{e.implication}
g \in \bigcup_n \cG_n \ \Rightarrow \
f_0 + g \in \lock \, .
\end{equation}
So in order to prove Theorem~\ref{t.brick} it is sufficient to show that
$\P_\bb(\cG_n) \to 1$ as $n \to \infty$.

Fix a positive integer $n$.
Let $d_n \coloneqq 2^n-1$ and let
$\tilde{\pi}_n \colon C^\ba(\Omega) \to \R^{d_n}$ be the map
$$
f \mapsto \big( c_\omega(f) \big)_{\omega \in \Omega^* \, , \, |\omega| < n}
$$
that associates to a function $f$ its Haar coefficients
of level less than $n$ (the initial coefficient $c(f)$ being disregarded).
Then $\tilde{\pi}_n(\cH_\bb)$ equals the box
$$
Q_n \coloneqq \prod_{|\omega| < n} [-b_\omega, b_\omega] \, ;
$$
moreover the push-forward of $\P_\bb$ under $\tilde{\pi}_n$
is the normalized Lebesgue measure on $Q_n$.
For each measure $\mu$ on $\Omega$,
define a linear map $I_\mu \colon \R^{d_n} \to \R$ by
$$
I_\mu \big( (x_\omega)_\omega \big) \coloneqq \sum_{\omega} \langle h_\omega, \mu \rangle x_\omega \, .
$$
Then $\langle A_n f, \mu \rangle = c(f) + I_\mu(\tilde\pi_n(f))$ for every $f \in C^\ba(\Omega)$.

Let $G_n \coloneqq \tilde{\pi}_n (\cG_n)$.
Note that $\cG_n = \tilde{\pi}_n^{-1} (G_n)$, because the definition of $\cG_n$ only involves the initial Haar coefficients.
Write $\cG_n^\mathsf{c} = \cH_\bb \setminus \cG_n$ and $G_n^\mathsf{c} = Q_n \setminus G_n$.
If $g \in \cG_n^\mathsf{c}$ then there exist distinct elements $\mu$, $\nu$ of $\cC_n$
such that
$$
|\langle A_n (f_0+g), \mu -\nu \rangle| \le \delta_n.
$$
Moreover, $\mu$, $\nu$ can be chosen as \emph{neighbors}
in the sense that $[\pi_n(\mu), \pi_n(\nu)]$ is an edge of the polytope $R_n$.
It follows that
$$
G_n^\mathsf{c} \subset \bigcup_{\substack{\mu \neq \nu \in \cC_n \\ \text{neighboors}}}
\left\{ x \in Q_n \st \big| (I_\mu-I_\nu)(x^{(0)} + x) \big| \le \delta_n \right\} \, ,
$$
where $x^{(0)} \coloneqq \tilde{\pi}_n(f_0)$.
In particular, by Lemma~\ref{l.slice},
\begin{equation}\label{e.ugly}
\P_\bb(\cG_n^\mathsf{c})
= \frac{\Leb_{d_n}(G_n^\mathsf{c})}{\Leb_{d_n}(Q_n)}
\le \frac{2 \delta_n}{\tau(Q_n)}
\sum_{\substack{\mu \neq \nu \in \cC_n \\ \text{neighboors}}} \|I_\mu - I_\nu\|_\infty^{-1} \, .
\end{equation}

Given $\mu \neq \nu$ in $\cC_n$,
there exists a node of the graph $G_n$ that is visited
by the cycle $C_n(\mu)$
but not by cycle $C_n(\nu)$.
That is, there exists a word $\omega$ of length $n-1$ such that
$$
\langle h_\omega, \mu \rangle = \frac{\pm 1}{2 \per(\mu)}
\quad \text{and} \quad
\langle h_\omega, \nu \rangle = 0 \, .
$$
We have $\per(\mu) \le 2^{n-1}$ (since the cycle $C_n(\mu)$ contains no repeated nodes),
and therefore $\|I_\mu - I_\nu\|_\infty \ge 2^{-n}$.
On the other hand, the number of pairs of neighbors $\{\mu,\nu\}$ in $\cC_n$ is
equal to the number of edges of the polytope $R_n$.
We have seen in \S\ref{ss.structure} that $R_n$ has no more than $2^{2^n}$ faces of all dimensions.
Substituting these estimates back in \eqref{e.ugly} we obtain:
\begin{equation}\label{e.uglier}
\P_\bb(\cG_n^\mathsf{c})
\le 2^{2^n + n + 1} \delta_n / \tau(Q_n) \, .
\end{equation}
Recall the evanescence and admissibility conditions \eqref{e.eva}, \eqref{e.adm},
which can be rewritten as:
$$
a_n = O \left( 2^{-2^{n+1}} a_{n-1} \right) \, , \qquad
\bar{b}_n = o(a_n) \, , \qquad
a_n = 2^{O(n)}  \underline{b}_n  \, .
$$
Then:
$$
\delta_n = \sum_{k=n}^\infty (k-n+1) O(a_k) = O(a_n) = O \left( 2^{-2^{n+1}} a_{n-1} \right) \, .
$$
Also,
$$
\bar{b}_n   = o(a_n)
            = O\left( 2^{-2^{n+1}}  a_{n-1} \right)
            = O\left( 2^{-2^{n+1} + O(n)}  \underline{b}_{n-1} \right) \, .
$$
In particular, if $n$ is large enough then $\min\{\underline{b}_0, \underline{b}_1, \dots , \underline{b}_{n-1} \} = \underline{b}_{n-1}$,
that is, the box $Q_n$ has thickness $\tau(Q_n) = \underline{b}_{n-1}$.
Substituting these estimates back in \eqref{e.uglier} we obtain:
$$
\P_\bb(\cG_n^\mathsf{c})
=
2^{2^n - 2^{n+1} + O(n)} a_{n-1} / \underline{b}_{n-1}
=
2^{-2^n +O(n)}
= o(1),
$$
As remarked before, it follows from \eqref{e.implication}
that $f_0 + g \in \lock$ for $\P_\bb$-a.e.\ $g$,
as we wanted to prove.
\end{proof}

\section{Directions for future research}\label{s.questions}

The results of this paper motivate a number of questions.

Consider the property of uniqueness of the maximizing measure
(which is of course weaker than the locking property).
This property is known to be topologically generic in many function spaces: see \cite{CLT,Jenkinson_survey}.
So it is natural to ask whether the uniqueness property is prevalent in function spaces
larger than those considered in this paper.

The most obvious open question is whether our prevalence results hold on larger function spaces.
The reader will have noticed that several estimates in the proof of Theorem~\ref{t.brick}
are very crude.
Nevertheless, it is not clear how the evanescence condition \eqref{e.eva}
could be significantly weakened.
One of the obstacles is the fact that
low-period vertices in $R_n$ belong to many edges:
for example the two vertices of period~$1$
are connected by edges to all other vertices.
Such ``edgy'' vertices give rise to many potential gaps to be controlled.
On the other hand, those vertices seem to be ``pointy''
and should typically generate large gaps, though we have not been able to formalize this.
On the other extreme we have the vertices of maximum period, namely the hamiltonian cycles (see Remark~\ref{r.hamiltonian}).
These vertices seem to be ``blunt''.
On the positive side it can be shown that each of them belongs to exactly $2^{n-1}$ edges,
which is the minimum number allowed by the dimension of the polytope.
To summarize, an improvement of our methods seems to require a finer understanding of the geometry and the combinatorics of the polytopes $R_n$ (or their duals).

Ultimately one would like to go beyond the locking property
and formalize the conjecture of Hunt and Ott~\cite{HO} (mentioned in Remark~\ref{r.decay} above)
that maximizing measures typically have \emph{low} period.

\medskip

\begin{fund}
J.B.\ was partially supported by projects Fondecyt 1140202 and
Anillo ACT1103 (Center of Dynamical Systems and Related Fields), and REDES 140138 (International Research Network on Dynamical Systems with Mild Hyperbolic Behaviour). Y.Z.\ was partially supported by project Fondecyt 3130622. 
\end{fund}

\begin{ack}
We thank Juan Rivera-Letelier for useful discussions and Congping Lin for assistance with numerical simulations.
We also thank the referee for a number of corrections and valuable suggestions, including simplification of a couple of arguments. Y. Zhang would like to thank IMPAN for the hospitality and financial support to attend the Simons Semester, during which this work was finalized.
\end{ack}


\medskip

\begin{small}
    \noindent
	\textsc{Jairo Bochi} \quad \email{\href{mailto:jairo.bochi@mat.puc.cl}{jairo.bochi@mat.puc.cl}} \quad \href{http://www.mat.uc.cl/~jairo.bochi}{www.mat.uc.cl/$\sim$jairo.bochi}
	
	\noindent \textsc{Yiwei Zhang} \quad \email{\href{mailto:yzhang@mat.puc.cl}{yzhang@mat.puc.cl}}

	\smallskip
	
    \noindent \textsc{Facultad de Matem\'aticas, Pontificia Universidad Cat\'olica de Chile. Avenida Vicu\~na Mackenna 4860, Santiago, Chile}

\end{small}

\end{document}